\documentclass[a4paper,12pt]{article}

\usepackage{a4,amssymb,amsmath,amsthm,url,xcolor,enumerate,float,lineno}
\usepackage{bezier,amsfonts,amssymb,graphicx,amsthm,url}
\usepackage[utf8]{inputenc}
\usepackage{subcaption}
\usepackage{amsmath}
\usepackage{amsfonts}
\usepackage{amssymb,amsthm}
\usepackage{url}
\usepackage{xcolor}
\usepackage{tikzit}
\newtheorem{theorem}{Theorem}[section]
\newtheorem{definition}{Definition}[section]
\newtheorem{proposition}[theorem]{Proposition}
\newtheorem{corollary}[theorem]{Corollary}
\newtheorem{lemma}[theorem]{Lemma}
\newtheorem*{theorem*}{Theorem}
\newtheorem*{problem*}{Problem}

\def\kaxxa{{\vcenter {\hrule height .2mm
\hbox{\vrule width .2mm height 2mm \kern 2mm
\vrule width .2mm} \hrule height .2mm}}}

\newcommand{\cay}[2]{\mathsf{Cay}\left(#1 ; #2\right)}

\tikzstyle{vertex}=[fill=black, draw=black, shape=circle, thick, scale=0.75]

\tikzstyle{edge}=[-, fill={rgb,255: red,0; green,128; blue,128}]
\tikzstyle{dir_edge}=[->]
\tikzstyle{blue_edge}=[-, draw=blue, ultra thick]
\tikzstyle{red_edge}=[-, draw=red, ultra thick]

\tikzset{every picture/.style={font issue=\footnotesize},
         font issue/.style={execute at begin picture={#1\selectfont}}
        }

\title{Flip colouring of graphs}
\author{\begin{tabular}{cc}Yair Caro & Josef Lauri \\ University of Haifa-Oranim & University of Malta\\\\ Xandru Mifsud & Raphael Yuster \\ University of Malta & University of Haifa \end{tabular}}

\date{Christina Zarb \\ University of Malta}

\DeclareGraphicsExtensions{.pdf,.png,.jpg}

\begin{document}

\maketitle

\begin{abstract}
It is proved that for integers $b, r$ such that $3 \leq  b < r \leq \binom{b+1}{2} - 1$, there exists a red/blue edge-colored graph such that the red degree of every vertex is $r$,  the blue degree of every vertex is $b$, yet in the closed neighborhood of every vertex there are more blue edges than red edges.
The upper bound $r \le \binom{b+1}{2}-1$ is best possible for any $b \ge 3$.
We further extend this theorem to more than two colours, and to larger neighbourhoods.

A useful result required in some of our proofs, of independent interest, is that for integers $r,t$ such that $0 \leq t \le  \frac{r^2}{2}  -  5r^{3/2}$, there exists an $r$-regular graph in which each open neighborhood induces precisely $t$ edges.

Several explicit constructions are introduced and relationships with constant linked graphs, $(r,b)$-regular graphs and vertex transitive graphs are revealed.
\end{abstract}
\section{Introduction}

Local versus global phenomena are widely considered both in graph theory (combinatorics in general) and in social sciences \cite{ABDULLAH20151,caro2018effect,chebotarev2023power,FISHBURN1986165,lerman2016majority,LISONEK1995153}.  %Interesting phenomena are those in which there is  either a counter-intuitive connection between "local property that  holds globally "  and a "global property", or when  there is a direct structural connection between "local property that holds globally"   and a "global property".  
Such phenomena occur in the most elementary graph theory observations as well as in highly involved theorems and conjectures.

A simple example dates back to Euler: Every degree is even (a local property), if and only if each component has an Eulerian circuit (a global property). More involved examples are Tur\'an-type problems \cite{furedi2013history,turan1941external,turan1954theory}, broadly showing that if a graph does not contain some fixed graph (local property), then its number of edges cannot be too large (a global property).
As an illustration of a somewhat counter-intuitive example, we recall a famous theorem of Erd\H{o}s \cite{erdos1959graph} stating that for any $h,k \ge 3$, there is a graph whose shortest cycle has length at least $h$ (a locally verifiable property), yet its chromatic number is at least $k$ (a global property).
Even more challenging is the Erd\H{o}s-Hajnal conjecture \cite{hajnal1989ramsey} stating that for every graph $H$, there is a constant $\lambda_H > 0$ such that if $G$ is an $n$-vertex graph not containing an induced copy of $H$, then $G$ contains a homogeneous set of order $\Omega(n^{\lambda_H})$.
Indeed one can find as many illustrious examples as one wishes.

We now introduce an ``umbrella'' for the local-global problem considered in this paper. Recall that $N^G(v)$ denotes the set of neighbors of a vertex $v$ in $G$
(its open neighborhood) and $N^G[V]= N^G(v) \cup \{v\}$ denotes the closed neighborhood. We omit the superscript when the graph is clear from context. Some further useful notation follows.
\begin{itemize}
\item{For a colouring $f : E(G) \rightarrow  \{ 1,\ldots,k \}$, $k\geq 1$, let $E(j) =  \{ e \in  E(G) : f( e )  = j  \}$ and $e_j(G) = |E(j)|$. For $k = 1$ we use $e(G)$.}
\item{For a vertex $v$, $\deg_j(v)  =  | \{ e :  e \mbox{  incident with $v$ and $f(e )  = j$}  \}|$.  For $k = 1$ we use $\deg(v)$.}
\item{For a vertex $v$, $e_ j[v]   =  | E(j) \cap  E(N[v]) |$.  For $k = 1$ we use $e[v]$.} 
\item{For a vertex $v$, $e_ j(v)   =  | E(j) \cap  E(N(v)) |$.  For $k = 1$ we use $e(v)$.}
\end{itemize}

We now state the general flip colouring problem: Given a graph $G$  and an integer $k \geq 2$, does there exist a colouring $f : E(G) \rightarrow \{ 1,\ldots,k \}$ such that:
\begin{itemize}
\item{for  every vertex $v$, $\deg_j(v) > \deg_i(v)$  for $1 \leq i< j \leq k$  (in particular, forcing  global majority $e_j(G)  > e_i(G)$  for $1 \leq i< j \leq k$)}\,;
 \item{for every vertex $v$, $ e_j[v] < e_i[v]$  for $1 \leq i< j \leq k$  (forcing an opposite local majority)\,.}  
\end{itemize}
 If such an edge-colouring exists, then $G$  is said to be a {\em $k$-flip graph} and the colouring $f$
 is a {\em $k$-flip colouring}.

We shall mostly deal with a version of this problem restricted to regular edge-coloured subgraphs (hence regular graphs), as this question already captures the essence of the problem and reduces notation overload. Namely,
given $k \geq 2$, a $d$-regular graph $G$ and a strictly increasing positive integer sequence $(a_1, \dots, a_k)$ such that $d = \sum_{j=1}^k a_j$, does there exist a colouring $ f : E(G) \rightarrow  \{ 1,\ldots,k \}$ such that:
\begin{itemize}
\item{$E(j)$ spans an $a_j$-regular subgraph, i.e, $\deg_j(v) = a_j$ for every $v\in V(G)$\,;}
\item{for every vertex $ v \in  V(G)$,  $e_k[v] < e_{k-1}[v] < \cdots <  e_1[v]$.}
\end{itemize}    
 
If such an edge-colouring exists then $G$ is said to be an {\em $(a_1,\ldots,a_k)$-flip graph} (in particular, a $k$-flip graph) and $(a_1, \dots, a_k)$ is called a {\em flip sequence} of $G$.
An illustrative example is given in Figure \ref{flip_4_3_smallest}.

\begin{figure}[ht!]
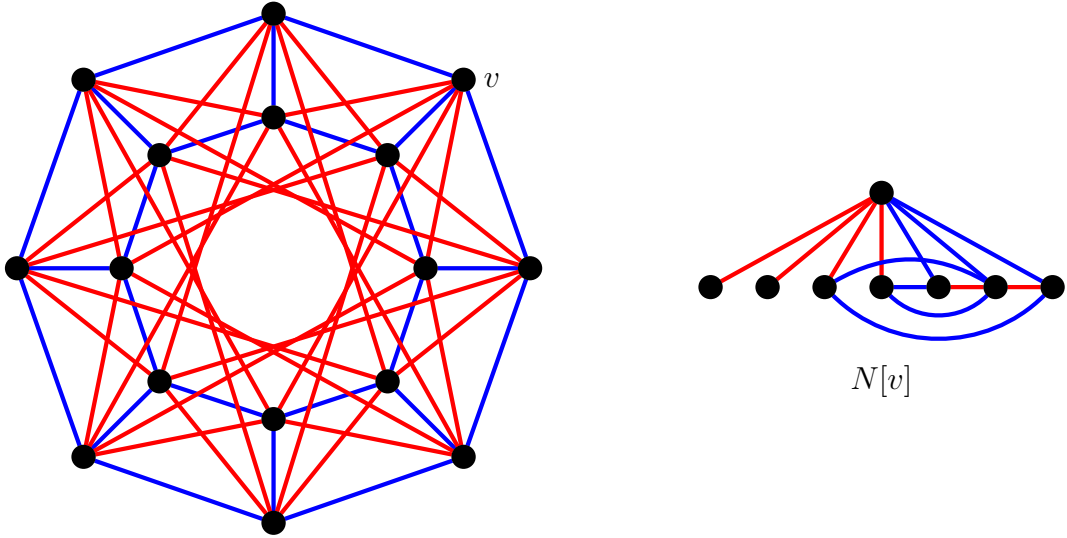

	\centering
	\ctikzfig{figures/flip_4_3_smallest}
	\caption{Smallest known $(3, 4)$-flip graph having 16 vertices, with the subgraph induced by the closed neighbourhood of any vertex $v$ illustrated on the right.}
	\label{flip_4_3_smallest}
\end{figure}

\subsection{Major problems concerning flip graphs}
 
The notion of flip graph gives rise to several natural problems:
\begin{enumerate}
	
\item
\emph{Characterise flip sequences:}	Given a strictly increasing positive integer sequence $(a_1, \dots, a_k)$, is it a flip sequence of some graph?
	
\item
\emph{Smallest order of an $(a_1,\ldots,a_k)$-flip graph:} Given a flip sequence $(a_1, \dots, a_k)$, determine the smallest order of a graph realising it.

\item
Devise explicit constructions for flip graphs with two or more colours.

\item
\emph{Interval flip.} Find $k$-flip sequences of the form $(1+t,2+t,\ldots,k+t)$.
	
\item
Extend the notion of flip from counting colour occurrences in closed neighbourhoods to counting colour occurrences in $t$-closed neighbourhoods (a vertex is in the $t$-closed neighbourhood of $v$ if its distance from $v$ is at most~$t$).

\item
\emph{Complexity of recognising a $k$-flip graph.}
Given a (possibly regular) graph $G$ and an integer $ k \geq 2$, determine whether it is a $k$-flip graph. The problem is clearly in NP, but is it NP-complete? 

\end{enumerate}
We shall consider most of these questions in the sequel. The paper is organised as follows.

Section 2 introduces the coloured Cartesian product technique which is used several times along this paper, together with a preliminary application using the family of $(r,c)$-constant graphs, which are $r$-regular graphs such that for every vertex $v$, the induced graph on $N(v)$ contains exactly $c$ edges.

Section 3 is about the flip problem with two colours. We develop techniques for constructing $(b,r)$-flip graphs using Cartesian products. We prove that a necessary and sufficient condition for $(b,r)$ to be a flip sequence is $3 \leq b < r \leq \binom{b+1}{2} -1$. This theorem completely answers Problems 1, 4 above for flip sequences of length two, supplies constructions as requested in Problem 3, and gives an upper bound for Problem 2.

Section 4 concerns the case of three or more colours. In particular, we prove that if $(a_1, a_2, a_3)$ is a flip sequence, then $a_3 \leq 2(a_1)^2$. We provide a construction that comes close to this bound.
Interestingly, it is revealed that when at least four colours are used, there are $k$-flip sequences where $a_k$ can be arbitrarily large even if $a_1$ is fixed.
 
Section 5 concerns $(r,c)$-constant graphs and their applications to Problem 4 (interval flip).
We prove the second theorem mentioned in the abstract and use it to prove that the interval $[b,\ldots, \frac{b^2}{4}  -  \frac{5b^{3/2}}{2}]$ is a flip sequence. We also propose a simple construction showing that for $b \geq 3$,  the interval $[b,\ldots,2b -2]$ is a flip sequence, which is useful with regards to the above result for small values of $b$.   

Section 6 concerns Problem 5. We prove several results concerning the extension of flip graphs to larger neighbourhoods.
 
Finally, section 7 summarises the current work and offers further open problems.
\section{The coloured Cartesian product technique}

\subsection{Cartesian products of edge-coloured graphs}

The Cartesian product of graphs will be useful in the construction of flip-graphs. Due to the additive nature of the degree and closed-neighbourhood sizes under the Cartesian product, this allows us to consider its factors independently. Before doing so, we recall the definition of Cartesian product and outline a number of its properties, including edge-colouring inheritance. 

\begin{definition}[Cartesian product] The Cartesian product $G \ \square \ H$ of the graphs $G$ and $H$ is the graph such that $V\left(G \ \square \ H\right) = V(G) \times V(H)$ and there is an edge $\{uv, u^\prime v^\prime\}$ in $G \ \square \ H$ if and only if either $u = u^\prime$ and $vv^\prime \in E(H)$, or $v = v^\prime$ and $uu^\prime \in E(G)$.
\end{definition}

The Cartesian product of graphs is commutative and associative, so the Cartesian product of a finite set of graphs is well-defined. It also enjoys a number of additional properties.
In particular, it is vertex-transitive if and only if each of its factors is vertex-transitive. More so, with an appropriate choice of generating set, the Cartesian product of Cayley graphs is also a Cayley graph. 

Let $G_1,\ldots,G_r$ be edge-coloured graphs. We extend the edge-colourings of the $G_i$'s to an edge-colouring of their Cartesian product in the natural way: the colour of $e = \{u_1u_2\cdots u_r, u_1^\prime u_2^\prime \cdots u_r^\prime\}$ is the colour of the unique edge
$u_iu_i^\prime \in E(G_i)$ for which $u_i \neq u_i^\prime$.
This colouring inheritance is illustrated in Figure \ref{CPDiagram}, with its properties summarised in Lemma \ref{CPColourLemma}.

\begin{figure}[ht!]
\centering
\includegraphics[width=0.43\textwidth]{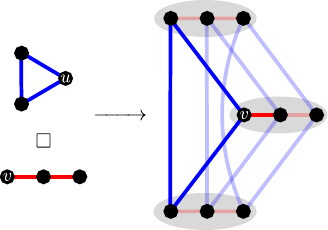}
\caption{Illustration of the edge-colouring inheritance in $K_3 \ \square \ P_3$ from its factors, where $K_3$ is coloured blue and $P_3$ is coloured red.}
\label{CPDiagram}
\end{figure}

\begin{lemma}\label{CPColourLemma}
	Let $G_1,\ldots,G_r$ be edge-coloured by colour set $[k] = \{1, \dots, k\}$.
	Let $G$ be their edge-coloured Cartesian product with the inherited colouring from its factors. Then
	for any $i \in [k]$ and $u_1\cdots u_r \in V(G)$:
	\begin{align*}
	\deg_i (u_1\cdots u_r) & = \sum_{j=1}^r \deg_i^{G_j}(u_j)\,,\\
	e_i\left[u_1\cdots u_r \right] & = \sum_{j=1}^r e_i^{G_j}[u_j]\,.
	\end{align*}
\end{lemma}

\subsection{$(r,c)$-constant graphs and their sub-families}

We recall constant link graphs , $(r,b)$-regular graphs, and introduce $(r,c)$-constant graphs.
{\em Constant link} graphs are those for which all subgraphs induced by open neighbourhoods are isomorphic
to some fixed graph $H$ (the link). Problems concerning which graphs $H$ can be links is an old (mostly unsolved) problem stated first by Zikov in 1964, which received attention over the years.  For some references see  \cite{blass1980trees,  conder2021parameters, https://doi.org/10.1002/jgt.3190090313, larrion2011small}.

Graphs are called {\em $(r ,b)$-regular} if they are $r$-regular and all open neighbourhoods induce a $b$-regular graph (hence $e[v] = r+\frac{br}{2}$). For a recent article and further references see \cite{conder2021parameters}.  

Generalising these two families, we define {\em $(r,c)$-constant graphs} to be $r$-regular graphs in which for every vertex $v$,  $e(v) = c$, hence $e[v] = r +c$.  

Figure \ref{RC_Const_Hierarchy} illustrates the hierarchy of several families of graphs in relation to $(r, c)$-constant graphs.

\begin{figure}[ht!]
\centering
\includegraphics[width=0.525\textwidth]{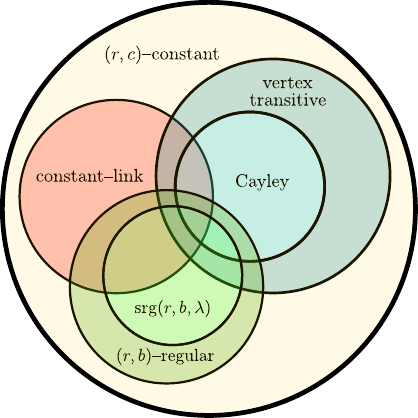}
\caption{Hierarchy of $(r, c)$-constant graphs and their sub-families of interest.}
\label{RC_Const_Hierarchy}
\end{figure} 

\subsection{Coloured Cartesian product lemmas}

This subsection explores the useful tool of coloured Cartesian products (abbreviated as CCP) via two lemmas. 
The first Lemma \ref{CCPL1} describes how to construct a $k$-flip graph from a family containing regular graphs satisfying prescribed conditions. The second Lemma \ref{CCPL2} describes how to construct $k$-flip graphs, $k$-flip sequences, $(r,c)$-constant graphs, $(r,b)$-regular and constant-link graphs, from such existing graphs, respectively. We also demonstrate some initial applications of these lemmas. 

\begin{lemma}[CCP Lemma I]\label{CCPL1}
Let $k \geq 2$ be an integer. Suppose $H_1, \dots, H_k$ are $a_1, \dots, a_k$ regular graphs respectively, with $a_i < a_j$ for $1 \leq i < j \leq k$. Furthermore, suppose that for $1 \leq i < k$,$$\max\limits_{u \in V(H_{i+1})} e [u] < \min\limits_{v \in V(H_i)} e [v]\,.$$ Then $G = \square_{j=1}^k H_j$ is a $(a_1, \dots, a_k)$-flip graph.
\end{lemma}
\begin{proof}
Colour the edges of $H_j$ using colour $j$, for $1 \leq j \leq k$. Let $G = \square_{j=1}^k H_j$ be the corresponding CCP. Clearly then for every vertex $w$ in $G$, given any $1 \leq i < k$,  $\deg_i (w) = a_i < a_{i+1} = \deg_{i+1} (w)$ and $$e_{i+1} [w] \leq \max\limits_{u \in V(H_{i+1})} e [u] < \min\limits_{v \in V(H_i)} e [v] \leq e_i [w]$$ and consequently $G$ is an $(a_1, \dots, a_k)$-flip graph.
\end{proof}

\begin{lemma}[CCP Lemma II] \label{CCPL2}
	Let $k \geq 2$ be an integer.
	\begin{enumerate}[i.]
		\item If, for $1 \leq j \leq q$, $(a_{j,1}, \ldots, a_{j,k})$ are $k$-flip sequences then
		$(a_1,\ldots,a_k)$ is a $k$-flip sequence where $a_i = \sum_{j=1}^q a_{j,i}$.
		\item If $H_j$ for $1 \leq j \leq q$ are $(r_ j, c_j)$-constant graphs, then there exists a graph $G$ which is a $\left(\sum_{j=1}^q r_j, \sum_{j=1}^q c_j\right)$-constant graph.
		\item If $H_j$ for $1 \leq j \leq q$ are $(r_j, b)$-regular graphs, then there exists a graph $G$ which is a $\left(\sum_{j=1}^q r_j, b\right)$-regular graph.
		\item If, for $1 \leq j \leq q$, $G_j$ is a constant-link graph with link $H_j$, then $\square_{j=1}^q G_j$ is a constant-link graph with link $\cup_{j=1}^q H_j$.
	\end{enumerate}
\end{lemma}
\begin{proof}
	We prove (i) for the case $q=2$; the other claims and the extension to larger $q$ follow the same outline. Let $(a_1, \dots, a_k)$ and $(b_1, \dots, b_k)$ be $k$-flip sequences, so there exists a graph $G$ which is a $(a_1, \dots, a_k)$-flip graph, and a graph $H$ which is a $(b_1, \dots, b_k)$-flip graph. 
	
	Consider the CCP graph $F = G \ \square \ H$. The colour degrees of a vertex $v$ in $V(F)$ are $c_i = a_i + b_i$ for $1 \leq i \leq k$ and note that $c_i < c_j$ for $1 \leq i < j \leq k$. Also for $v=(x,y) \in V(F)$, we have that
	$$
	e_i [v] = e_i[x] + e_i[y] > e_j[x] + e_j[y] = e_j[v]
	$$
	for $1 \leq i < j \leq k$. Hence $F$ is a $(c_1, \dots, c_k)$-flip graph.
\end{proof}

Regular graphs with constant link are useful for flip graph construction, as if $G$ is $b$-regular with constant link $H$ where $e(H) = c$, then for any  $r$ such that $b < r < b +c$,  and for any triangle free $r$-regular graph $F$, the CCP graph $G \ \square \ F$, with $G$ coloured blue and $F$ coloured red, is an $(r+b)$-regular graph with $r$ red edges and $b$ blue edges incident with every vertex $v$ and yet in $N[v]$ there are just $r$ red edges but $b+c > r$ blue edges. Hence $G \ \square \ F$ is a $(b,r)$-flip graph.

Similarly, $(b,c)$-regular graphs are useful for flip graph construction, as if $G$ is a $(b,c)$-regular graph and $b < r < b+\frac{bc}{2}$  then for any triangle free $r$-regular graph $F$, the CCP graph $G \ \square \ F$, with $G$ coloured blue and $F$ coloured red, is an $(r+b)$-regular graph with $r$ red edges and $b$ blue edges  incident with every vertex $v$ and yet in $N[v]$ there are just $r$ red edges but  $ b +\frac{bc}{2}  > r$  blue edges. Hence $G \ \square \ F$ is a $(b,r)$-flip graph.

Clearly if $G$ is a $(b,c)$-constant regular graph and $b < r < r +c$, then, as above, for any $r$ such that $b < r <   b +c$,  and for any triangle free  $r$-regular graph $F$, the CCP graph $G \ \square \ F$ is a $(b,r)$-flip graph.

Therefore, we have an abundance of known $(b,c)$-regular graphs, Cayley graphs, vertex-transitive graphs and graphs with constant link as well as strongly regular graphs, which can be used to construct a panoramic collection of flip-graphs via the coloured Cartesian product.

\section{Existence of $(b,r)$-flip graphs}

This section settles the question: which pairs $(b,r)$ form a flip sequence?
We also establish an upper bound for $h(b,r)$, the smallest possible order of a $(b,r)$-flip graph.

\begin{theorem}\label{RBClassification}
	Let $r, b \in \mathbb{N}$. If $3 \leq b < r \leq \binom{b+1}{2} - 1$ then there exists a $(b,r)$-flip graph, and both the upper bound and lower bound are sharp.
\end{theorem}

We split the proof into two parts. We first prove that we must have $r < \binom{b+1}{2}$
(in particular, since $r > b$, this implies $b \ge 3$). Then we prove the rest of the theorem.
Before proceeding further, we need some notation describing edge-coloured triangles rooted at some vertex.

Let $B=blue=1$ and $R=red=2$. In a graph $G$ with edges coloured from $\{B, R\}$ and $X, Y, Z \in \{B,R\}$ a triangle rooted at a vertex $v$ is said to be of \textit{type} $XYZ$ \textit{at} $v$ if the two edges incident to $v$ are coloured $X$ and $Y$, and the third edge is coloured $Z$ (the types $BRR$ and $RBR$ are considered identical and the types $RBB$ and $BRB$ are also considered identical).
Hence, a triangle rooted at $v$ can have one of six possible types, illustrated in Figure \ref{TriangleTypes}.
\begin{figure}[h!]
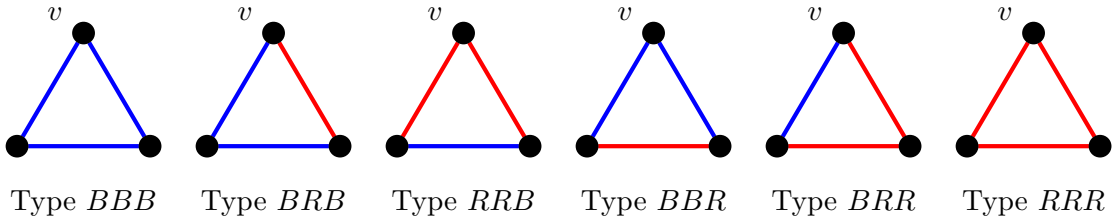

\small
\ctikzfig{figures/triangle_two_labelings}
\caption{All possible triangle types.}
\label{TriangleTypes}
\end{figure}

Let $T_{XYZ}(v)$ be the number of triangles of type $XYZ$ rooted at $v$. We need the following simple lemma.

\begin{lemma} \label{rbLemma1}
In a graph $G$ with edges coloured from $\{B, R\}$, we have that
	\begin{eqnarray*}
		2 \sum_{v \in V} T_{RRB}(v) = \sum_{v \in V} T_{BRR}(v) & \mathrm{and} & 2 \sum_{v \in V} T_{BBR}(v) = \sum_{v \in V} T_{BRB}(v)\;.
	\end{eqnarray*}
\end{lemma}

\begin{proof}
	Consider a triangle with two edges coloured $R$ and a single edge coloured $B$.
	Each such triangle is of $BRR$-type for two vertices and is of $RRB$-type for a single vertex.
	Summing over all such triangles yields the first equality. The second equality is symmetrical by considering triangles with two edges coloured $B$ and a single edge coloured $R$.
\end{proof}

\begin{proposition} \label{rb_upperBoundThm}
	In a $(b,r)$-sequence, we must have $r < \binom{b + 1}{2}$.
\end{proposition}

\begin{proof}
	Suppose that $G=(V,E)$ is a $(b,r)$-flip graph equipped with a suitable colouring.
	Note that since $e_B[v] = e_B(v) + b$ and $e_R[v] = e_R(v) + r$, we have from the definition of a flip graph that $e_B(v)-e_R(v) > r-b$.
	For any vertex $v$, the number of edges coloured $B$ (respectively $R$) in the open neighbourhood is equal to the number of triangles rooted at $v$ which have an edge coloured $B$ (respectively $R$) not incident to $v$, so:
	\begin{align*}
	e_B(v) & = T_{BBB}(v) + T_{BRB}(v) + T_{RRB}(v)\,,\\
	e_R(v) & = T_{RRR}(v) + T_{BBR}(v) + T_{BRR}(v)\,.
	\end{align*}
Using these equalities and Lemma \ref{rbLemma1}, we obtain:
	\begin{align*}
		~ & \sum_{v \in V} e_B(v) - e_R(v)\\
		& = \sum_{v \in V} T_{BBB}(v) + T_{BRB}(v) + T_{RRB}(v) - T_{RRR}(v) - T_{BBR}(v) - T_{BRR}(v) \\
		&\leq \sum_{v \in V} T_{BBB}(v) + \left(\sum_{v \in V} T_{BRB}(v) - T_{BBR}(v)\right) + \left(\sum_{v \in V} T_{RRB}(v) - T_{BRR}(v)\right) \\
		&= \sum_{v \in V} T_{BBB}(v) + \sum_{v \in V} T_{BBR}(v) - \sum_{v \in V} T_{RRB}(v) \\
		&\leq \sum_{v \in V} T_{BBB}(v) + T_{BBR}(v)
	\end{align*}
	where the next to last step follows from Lemma \ref{rbLemma1}.
	Observe that $T_{BBB}(v) + T_{BBR}(v)$ is the number of edges incident with two blue neighbours of $v$ hence is at most $\binom{b}{2}$, thus
	$$
	\sum_{v \in V} e_B(v) - e_R(v) \le |V|\binom{b}{2}\;.
	$$
	It follows that there exists a vertex $v \in V$ such that $e_B (v) - e_R (v) \leq \binom{b}{2}$.
	Recalling that $r-b < e_B (v) - e_R (v)$  we have that $r-b < \binom{b}{2}$, so $r < \binom{b+1}{2}$.
\end{proof}

\begin{proof}[Proof of Theorem \ref{RBClassification}]
That $3 \leq b < r \leq \binom{b+1}{2} - 1$ follows immediately from Proposition \ref{rb_upperBoundThm}.
We show that given such $r$ and $b$, a $(b,r)$-flip graph exists. 

Consider the CCP graph $G = K_{r,r} \ \square \ K_{b+1}$ where the edges of $K_{b+1}$ are coloured blue and the edges of $K_{r, r}$ are coloured red.
By Lemma \ref{CPColourLemma}, it follows that every vertex $v$ in $G$ has $\deg_B (v) = b$ and $\deg_R (v) = r$. Moreover, we have that $e_B [v] = \binom{b+1}{2}$ and $e_R [v] = r$.  Hence $G$ is a $(b, r)$-flip graph.
\end{proof}

\subsection{Upper bounds on $h(b, r)$}

As the graph $G$ in the proof of Theorem \ref{RBClassification} has $2r(b+1)$ vertices, it follows that $h(b, r) \le 2r(b+1)$. On the other hand, we have already seen in Figure \ref{flip_4_3_smallest} an example of a $(3,4)$-flip graph with $16$ vertices. Hence $h(3,4) \leq 16$, so the aforementioned bound is not tight.
The next result offers a more general construction, considerably improving the $2r(b+1)$ upper bound.

\begin{theorem} \label{upperbound_hrb}
Let $b, r \in \mathbb{N}$ such that $3 \le b < r \le \binom{b+1}{2}-1$. Then,
$$h(b, r) \leq \min \left\{2(r +x)(b +1 - x) \colon  x \in \mathbb{Z},~ 0 \leq x \leq b,~ x + \binom{b +1 - x}{2} > r\right\}\,.$$	
\end{theorem}
\begin{proof}
Let $x$ be an integer satisfying the theorem's condition. Consider an edge-colouring of
$K_{r+x, r+x}$ such that an $x$-factor is coloured $B$ and an $r$-factor is coloured $R$. Also consider $K_{b+1-x}$ where all the edges are coloured $B$.
Every vertex $v$ of the CCP graph $G=K_{r+x, r+x} \ \square \ K_{b+1-x}$ has $\deg_B (v) = b - x + x = b$ and $\deg_R (v) = r$. Moreover, $e_B[v] = x + \binom{b +1 - x}{2}$ and $e_R [v] = r$. 
By our choice of $x$, it follows that $e_B[v] > e_R [v]$.
\end{proof}

The upper bound in Theorem \ref{upperbound_hrb} warrants further analysis. We first require a lemma.
\begin{lemma}\label{minima_hrb}
Let $b, r \in \mathbb{N}$ such that $3 \le b < r \le \binom{b+1}{2}-1$. Let
$x_0 = \lceil b - (1+\sqrt{1+8(r-b)})/2\rceil - 1$. Then, $2(r+x_0)(b+1-x_0)$
\begin{align*}
	&\ \ = \min \left\{2(r +x)(b +1 - x) \colon  x \in \mathbb{Z},~ 0 \leq x \leq b, ~x + \binom{b +1 - x}{2} > r\right\}\,.
\end{align*}
\end{lemma}
\begin{proof}
Let $g(z) = z + \binom{b+1-z}{2} - r$ be a real-valued function. As $2(r+x)(b+1-x)$ is strictly decreasing in $[0,\infty)$, the claimed minimum, in integer value, is attained for the largest possible integer $0  \leq x_0 \leq b$, such that $g(x_0) > 0$. Rearranging, $g(z)$ can be written as a quadratic in $z$, $$g(z) = \binom{b+1}{2} - r + \left(\frac{1}{2} - b\right) z + \frac{z^2}{2}$$ which has a minimum, as well as distinct roots $z_\pm = b - \frac{1 \mp \sqrt{1+8(r-b)}}{2}$. 

Then $g(z) > 0$ whenever $z < z_-$ or $z > z_+$. Since the integer $x_0$ we are seeking must satisfy $x_0 \leq b$, then the only admissible case when $g(x_0) > 0$ is when $x_0 < z_-$. Since we seek the largest such integer, then $x_0 = \left\lceil b - \frac{1+\sqrt{1+8(r-b)}}{2}\right\rceil - 1$. 
What remains is to ensure that $x_0 \geq 0$. It suffices to show that $b > \frac{1+\sqrt{1+8(r-b)}}{2}$. Rearranging, we require that $\frac{(2b - 1)^2 - 1}{8} > r-b$. Indeed,
$$\frac{(2b - 1)^2 - 1}{8} = \frac{b^2}{2} - \frac{b}{2} = \binom{b+1}{2} - b > r - b\,.$$
Hence $x_0$ is the largest integer for which the minimum is attained, as required.
\end{proof}

Substituting the value $x_0$ obtained in Lemma \ref{minima_hrb} and using Theorem \ref{upperbound_hrb} we obtain
the following easily verified corollary:
\begin{corollary}\label{hrb_best_bound}
Let $b, r \in \mathbb{N}$ such that $3 \leq b < r \le \binom{b+1}{2}-1$. Then
$$ h(b,r) \leq 2 \left(r + b + 1 - \left\lfloor \dfrac{5 + \sqrt{1+8(r-b)}}{2} \right\rfloor\right) \left\lfloor \dfrac{5 + \sqrt{1+8(r-b)}}{2} \right\rfloor\,.
$$	
\end{corollary}

%\begin{proof}
%	Substituting in Theorem \ref{upperbound_hrb} for Lemmas \ref{monotone_hrb} and \ref{minima_hrb}, we get that
%	\begin{align*}
%	&h(b,r)\\
%	&\leq 2 \left(r - 1 + \left\lceil b - \dfrac{1+\sqrt{1+8(r-b)}}{2}\right\rceil\right)\left(b+2 - \left\lceil b - \dfrac{1+\sqrt{1+8(r-b)}}{2}\right\rceil\right)\\
%	&= 2 \left(r + b - 1 + \left\lceil - \dfrac{1+\sqrt{1+8(r-b)}}{2}\right\rceil\right)\left(2 - \left\lceil - \dfrac{1+\sqrt{1+8(r-b)}}{2}\right\rceil\right)\\
%	&= 2 \left(r + b - 1 - \left\lfloor  \dfrac{1+\sqrt{1+8(r-b)}}{2}\right\rfloor\right)\left(2 + \left\lfloor \dfrac{1+\sqrt{1+8(r-b)}}{2}\right\rfloor\right)\\
%	&= 2 \left(r + b + 1 - \left\lfloor \dfrac{5+\sqrt{1+8(r-b)}}{2}\right\rfloor\right)\left\lfloor \dfrac{5+\sqrt{1+8(r-b)}}{2}\right\rfloor
%\end{align*}
%as required.
%\end{proof}

Notice that for the case $b = r - 1$, valid for all $r \ge 4$, the minimum is obtained at $x_0 = r - 4$. The flip graph $G = K_{2r - 4, 2r - 4} \ \square \ K_{4}$ outlined in Theorem \ref{upperbound_hrb} is an $(r, r-1)$-flip graph with $16r - 32$ vertices. Even so, $(r-1, r)$-flip graphs exist with fewer than $16r - 32$ vertices as seen from the $(3, 4)$-flip graph in Figure 1.

%\textcolor{blue}{[Xandru] \textbf{Insert general Cayley construction - need to obtain from C+J.}} \textcolor{red}{\textbf{Is this the place where to define Cayley graphs?} \textit{I don't think it's appropriate...}}

\subsection{Weak flip graphs}

While $(2, r)$-flip graphs do not exist, weakening the flip constraint from $e_1[v] > e_2[v]$ to $e_1 [v] \geq e_2[v]$ (yet still requiring $r > b$), an admissible colouring can be found such that $E(1)$ spans a $b$-regular subgraph and $E(2)$ spans an $r$-regular subgraph. We term such graphs $(b,r)$-\textit{weak}-flip graphs. 

\begin{figure}[h!]
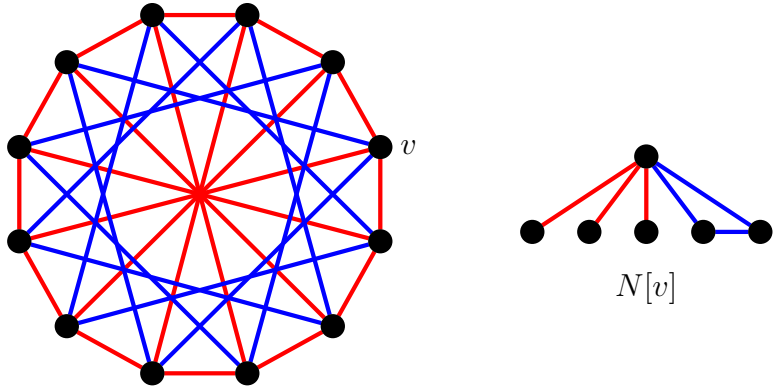

	\centering
	\ctikzfig{figures/weak_flip_3_2_smallest}
	\caption{Smallest known $(2, 3)$-weak-flip graph having $12$ vertices, with the subgraph induced by the closed neighbourhood of any vertex $v$ on the right.}
	\label{3_2_weak_flip}
\end{figure}

The CCP graph $K_{3,3} \ \square \ K_3$, with $K_{3,3}$ coloured red and $K_3$ coloured blue, is a $(2, 3)$-weak-flip graph on $18$ vertices. Figure \ref{3_2_weak_flip} illustrates the existence of a smaller $(2,3)$-weak-flip graph, having just $12$ vertices.

Likewise, $(1,r)$-flip graphs do not exist. However, not even $(1, r)$-weak-flip graphs exist. 

\begin{proposition}
Let $r \in \mathbb{N}$ such that $r > 1$. There is no $(1, r)$-weak-flip graph. 	
\end{proposition}

\begin{proof}
	Consider first the case $r > 2$. Suppose that a $(1, r)$-weak-flip graph $G$ exists. Let $v$ be a vertex in such a graph. Then $e_1[v] \leq 1 + \lfloor\frac{r}{2}\rfloor < r \leq e_2[v]$ since $r > 2$. 
	In the case when $r = 2$, suppose that a $(1, 2)$-weak-flip graph $G$ exists. Then by the above argument, it follows that $e_1 [v] = e_2 [v] = 2$.
	Consider a vertex $v$ with neighbours $u, w$ and $x$ such that $\{v, w\}$ and $\{v, u\}$ are coloured $2$ and $\{v, x\}$ is coloured $1$. Clearly $\{u, w\}$ must be coloured $1$ since $e_1[v] = e_2[v] = 2$. But then $u$ has some neighbour $y$ different from $v$ and $w$, such that $\{u, y\}$ is coloured $2$. Hence $e_2 [u] \geq 3$, a contradiction.
\end{proof}
\section{Flipping with three or more colours}

Unlike the case of two colours where $(b,r)$-flip sequences are completely characterised in Theorem \ref{RBClassification}, for $k \geq 3$ colours we do not have a full characterisation of all $k$-flip sequences.
 Our first result, nonetheless, establishes a necessary condition for $3$-flip sequences.  

\begin{theorem}\label{3colour_bound}
If $(a_1, a_2, a_3)$ is a $3$-flip sequence, then $a_3 < 2(a_1)^2$. 
\end{theorem}

\begin{proof}
Suppose on the contrary that $(a_1, a_2, a_3)$ is a flip sequence realised by some graph $G$, but that $2a_1^2 \le a_3$. We shall prove that for some vertex $v$ of $G$, $e_1 [v] \leq e_2[v]$ or $e_1 [v] \leq e_3 [v]$. 

For $i \in \{1,2,3\}$, let $N_i(v) = \left\{u \in V \colon \{u, v\} \in E(i) \right\}$.
Clearly, there is a set of (at least) $e_1[v] - (a_1)^2$ edges of $N[v]$ coloured $1$, having both of their endpoints in $N_2(v) \cup N_3 (v)$. We may assume $e_1 [v] - (a_1)^2 \geq 0$, otherwise we are done.

Now each of these $e_1 [v] - (a_1)^2$ edges forms a unique triangle with $v$, where only one edge of the triangle is coloured $1$. The other two edges of such a triangle contribute two edges coloured using either $2$ or $3$, to some open neighbourhood. This means that the total number of edges coloured using $2$ or $3$ in all open neighbourhoods is at least $\sum_{v \in V} 2\left(e_1 [v] - (a_1)^2\right)$. Hence there is some vertex $v$ with at least $2\left(e_1 [v] - (a_1)^2\right)$ edges coloured using $2$ or $3$ in its open neighbourhood.
Therefore, $$e_2 [v] + e_3[v] \geq a_3 + a_2 + 2e_1[v] - 2(a_1)^2\,.$$
But $a_3 + a_2 + 2e_1[v] - 2(a_1)^2 \geq 2 e_1[v]$ since $a_3 + a_2 - 2(a_1)^2 \geq a_3 - 2(a_1)^2 \geq 0$.

Hence $e_3[v] + e_2[v] \geq 2e_1 [v]$. But this means that $e_1 [v] \leq e_2[v]$ or $e_1 [v] \leq e_3 [v]$, which is a contradiction since $G$ is a flip graph. 
\end{proof}

In view of Theorem \ref{3colour_bound}, it is of interest to find construction of $3$-flip sequences in
which $a_3$ is quadratic in $a_1$. To this end, we use the following proposition.

\begin{proposition}\label{3colour_construction}
Let $a_1, a_2, a_3 \in \mathbb{N}$ such that $a_1 < a_2 < a_3$.
\begin{enumerate}[i.]
	\item If $H$ is an $(a_2, a_3)$-flip graph for which $e_2[v] < \binom{a_1 + 1}{2}$ for each vertex $v$, then $(a_1, a_2, a_3)$ is a $3$-flip sequence.
	\item Suppose that $H_1, H_2$ and $H_3$ are, respectively, $a_1, a_2$ and $a_3$ regular. If for $i \in \{1, 2\}$, we have that $$\max\limits_{u \in V(H_{i+1})} e [u] < \min\limits_{v \in V(H_i)} e [v]$$ then $(a_1, a_2, a_3)$ is a $3$-flip sequence.
\end{enumerate}	
\end{proposition}

\begin{proof}
	To prove (i), consider the CCP graph $G = H \ \square \ K_{a_1 + 1}$ where $H$ is coloured using $2$ and $3$, and $K_{a_1 + 1}$ is coloured using $1$. Then $G$ is a graph with colour degrees $a_i$ for colour $i \in \{1, 2, 3\}$. Moreover, $e_3 [v] < e_2 [v] < \binom{a_1 + 1}{2} = a_1 [v]$. Hence $G$ is an $(a_1, a_2, a_3)$-flip graph and we are done.
	
	For (ii), consider the CCP graph $G = H_1 \ \square \ H_2 \ \square \ H_3$, where $H_i$ is coloured using $i$ for $i \in \{1, 2, 3\}$. The result immediately follows by the conditions necessitated in (ii).
\end{proof}

Constructing $3$-flip sequences $(a_1, a_2, a_3)$ using Proposition \ref{3colour_construction} is relatively easy when $a_2$ and $a_3$ are not too far apart. For example, through Theorem \ref{upperbound_hrb}, we have seen the existence of $(r-1, r)$-flip graphs with $e_1[v] = r + 2$ and $e_2[v] = r$. 
Hence with $a_2 = r - 1, a_3 = r$ and $a_1 < a_2$ such that $r + 2 < \binom{a_1 + 1}{2}$, it follows that $(a_1, a_2, a_3)$ is a $3$-flip sequence. Note that we can take $a_1=\lceil \sqrt{2r} \rceil$ for $r \ge 9$.
This shows hat we can have $a_3 \approx 0.5(a_1)^2$, so Theorem \ref{3colour_bound} is tight up to a constant factor.

\subsection{Unbounded gap $k$-flip sequences}

We have seen that for two or three colours, the largest element in a flip sequence must be bounded above quadratically in the smallest element. A natural question is whether this extends to four or more colours. We shall see that the answer is negative -- indeed, we shall see constructions where the smallest element is constant and yet the largest element may be arbitrarily large.

We must first recall another handy graph product, the \textit{strong product}, along with a way to inherit an edge colouring from its factors.

\begin{definition}[Strong product] The strong product $H \boxtimes K$ of two graphs $H$ and $K$ is the graph such that $V\left(H \boxtimes K\right) = V(H) \times V(K)$ and there is an edge $\left\{(u, v), (u^\prime, v^\prime)\right\}$ in $H \boxtimes K$ if and only if either $u = u^\prime$ and $vv^\prime \in E(K)$, or
	$v = v^\prime$ and $uu^\prime \in E(H)$, or $uu^\prime \in E(H)$ and $vv^\prime \in E(K)$.
\end{definition}

Let $H$ and $K$ be two graphs with an edge-colouring from a set of colours $C$. We extend the edge-colourings of $H$ and $K$ to an edge-colouring of $H \boxtimes K$ as follows.
Consider edge $e = \left\{(u, v), (u^\prime, v^\prime)\right\}$ in $H \boxtimes K$; if $u = u^\prime$ then $e$ inherits the colouring of the edge $vv^\prime$ in $K$, otherwise if $u \neq u^\prime$ the colouring of the edge $vv^\prime$ in $H$ is inherited. Note that the inherited colourings of $H \boxtimes K$ and $K \boxtimes H$ are different, even though the two uncoloured graphs are isomorphic.

\begin{figure}[ht!]
\centering
\includegraphics[width=0.55\textwidth]{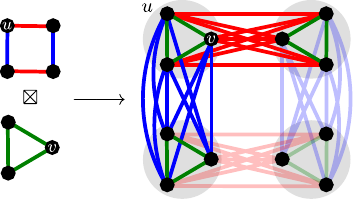}
\caption{Illustration of the edge-colouring inheritance in $K_{2, 2} \boxtimes K_3$ from its factors, where $K_{2,2}$ has red and blue coloured $1$-factor, and $K_3$ is coloured green. The edges in the closed neighbourhood of $(u,v)$ are highlighted.}
\label{SPDiagram}
\end{figure}

We are now in a position to prove the main result of this section.

\begin{theorem} \label{arbitraryGapsThm}
	Let $k \in \mathbb{N}, k > 3$. Then there is some constant $m = m(k) \in \mathbb{N}$ such that for all $N \in \mathbb{N}$, there exists a $k$-flip sequence $\left(a_1, a_2, \dots, a_k\right)$ such that $a_1 = m$ and $a_k > N$. 
\end{theorem}

\begin{proof}
Let $K$ be the complete graph $K_{2n}$ where $n > \dfrac{k(k^2 - 2k + 1)}{4(k-3)}$. Since $K$ is a complete graph on an even number of vertices, $K$ has a $1$-factorisation. For $1 < i < k$, let $k-i$ $1$-factors be coloured using colour $i$ and let the remaining edges be coloured $1$. It follows then that every vertex $v$ in $K_{2n}$ has $$\deg^K_1 (v) = 2n - 1 - \binom{k-1}{2}$$ incident edges coloured $1$ and $\deg^K_i (v) = k - i$ incident edges coloured $i$ for $1 < i < k$. For convenience, define $\deg^K_k (v) = 0$. Observe that the sequence $\deg^K_1 (v), \dots, \deg^K_k (b)$ is strictly decreasing, noting that by the choice of $n$,
$$
\deg^K_1 (v) - \deg^K_2 (v) > \dfrac{k(k-1)}{k-3} > 0\,.
$$
Since $K$ is a complete graph and each vertex $v$ has the same number of incident edges coloured $i$, then $e_i^K [v] = n \deg^K_i (v)$ for $1 \leq i \leq k$.

We now show that for every vertex $v$ in $K$, $(k-1)(e_1^K[v] - e_2^K[v]) > 4n^2$. Rearranging and substituting for $e_1^K[v]$ and $e_2^K[v]$ in terms of $n$ and $k$, we must show that $n > \dfrac{k(k^2 - 2k + 1)}{4(k-3)}$. This follows by our choice of $n$.
 
Consider $t \in \mathbb{N}$ such that
$$
t \geq \dfrac{4 n^2}{(k-1)\min\limits_{1 < i < k}\{e^K_i[v] - e^K_{i+1}[v]\}} = \dfrac{4n}{k-1}\,.
$$
Let $\rho = \dfrac{(k-1)(2t + k - 2)}{2}$ and let $H$ be a $\rho$-regular bipartite graph. For $0 \leq i \leq k - 2$, let $t + i$ matchings of $H$ be coloured using colour $2 + i$.

Let $G$ be the graph $H \boxtimes K$, inheriting the edge-colourings of $H$ and $K$ respectively. For a vertex $u$ in $H$, let $G_u$ be the subgraph of $G$ induced by the vertices $\left\{(u, w) \colon w \in V(K)\right\}$ in $G$; note that $G_u$ is isomorphic to $K$. 

If an edge $ux$ in $H$ is coloured $i$, then all the edges in $G$ between $V(G_u)$ and $V(G_x)$ are coloured $i$, by the inherited edge-colouring. Hence, if $u$ in $H$ and $v$ in $K$ have $\deg_i^H (u)$ and $\deg_i^K (v)$ incident edges coloured $i$, respectively, then the vertex $(u, v)$ in $G$ has $\deg_i^K (u) + |V(K)| \deg_i^H (v)$ incident edges coloured $i$. Consequently, by construction each vertex $(u, v)$ in $G$ has:
$$\deg^G_i \big((u, v)\big) = 
\begin{cases}
2n - 1 - \binom{k-1}{2} & i = 1 \\
(k - i) + 2n(t + i - 2) & 2 \leq i \leq k
\end{cases}$$
which is strictly increasing.

Now, consider the vertex $(u, v)$ in $G$. Let $u_1, \dots, u_\rho$ be the neighbours of $u$ in $H$. In $G$ we have that every vertex in $G_u$ has a neighbour in $G_{u_i}$ for $1 \leq i \leq \rho$, but for $1 \leq i < j \leq \rho$, we have that there are no edges between $G_{u_i}$ and $G_{u_j}$ in $G$, since $H$ is bipartite and $u_i$ and $u_j$ belong to the same partite set.
Therefore, the edges coloured $i$ in the closed neighbourhood of $(u, v)$ are:
\begin{enumerate}[i.]
	\item  those edges coloured $i$ in the closed neighbourhood of $v$ in each of the $\rho + 1$ copies of $K$ (namely $G_u, G_{u_1}, \dots, G_{u_\rho}$), totalling $(\rho + 1) e_i^K [v]$ edges,
	\item and a further $e_i^H [u]|V(K)|^2$ edges from the matchings coloured $i$ in $H$.
\end{enumerate}

This is clearly exemplified in Figure \ref{SPDiagram}. By construction of $H$ and $K$, it follows that for any vertex $(u, v)$ of $G$,
$$e_i^G \left[(u, v)\right] = 
\begin{cases}
(\rho + 1) e^K_1 [v] & i = 1 \\
(\rho + 1) e^K_i [v] + 4n^2(t + i - 2) & 2 \leq i \leq k
\end{cases}$$
which we now show to be strictly decreasing. 

Firstly note that since $\rho + 1 = (k-1)t + \binom{k-1}{2} + 1$, then there exists $\kappa \in \mathbb{R}$ such that $\kappa > 1$ and $\rho + 1 = (k-1)t\kappa$. Now, recall that $K$ has the property that $(k-1)(e_1^K [v] - e_2^K [v]) > 4n^2$.

Since $\kappa > 1$, it follows that $$(\rho + 1)\left(e_1^K [v] - e_2^K [v]\right) = (k-1)\left(e_1^K [v] - e_2^K [v]\right)(t\kappa) > 4n^2 t$$ and therefore since $e_1^K [v] > e_2^K [v]$, we have that $(\rho + 1) e_1^K [v] > (\rho + 1) e_2^K [v] + 4n^2 t$. Consequently, $e_1^G \left[(u, v)\right] > e_2^G \left[(u, v)\right]$ as required.

Now consider $2 \leq i \leq k - 1$. By the choice of $t$ and $\kappa > 1$, we have that 
\begin{align*}
(\rho + 1)\left(e_i^K [v] - e_{i+1}^K [v]\right) &= (k-1)\left(e_i^K [v] - e_{i+1}^K [v]\right)(t\kappa)\\
 &> 4n^2 \\
 &= 4n^2(t + i - 1) - 4n^2(t + i - 2)
\end{align*}
and therefore $e_i^G \left[(u, v)\right] > e_{i+1}^G \left[(u, v)\right]$.

It follows that $G$ is a flip graph on $k > 3$ colours, such that for any vertex $(u, v)$, $\deg_1 \big((u,v)\big)$ is only dependent on $k$ and $\deg_k \big((u,v)\big)$ increases with $t$. Since $t$ is not bounded from above in the construction, then given any $N \in \mathbb{N}$, a sufficiently large $t$ can be found such that $\deg_k \big((u,v)\big) > N$.
\end{proof}
\section{$(r,c)$-Constant graphs, long flipping intervals, and sufficient conditions for $k$-flip sequences}

This section explores the notion of $(r,c)$-constant graphs, first introduced in Section 2, and their use
in constructing long flipping intervals. This will then allow us to deduce a sufficient condition for $k$-flip sequences.

\subsection{Existence of $(r, c)$-constant graphs}

We recall that an $(r,c)$-constant graph $H$ is an $r$-regular graph such that for every vertex $v \in V(H)$, $e(v) = c$. We have already seen, as summarised in Figure \ref{RC_Const_Hierarchy}, that many familiar classes of graphs are subclasses of $(r,c)$-constant graphs. 

An inevitable problem regarding $(r,c)$-constant graphs is: given a positive integer $r$, for which integers $c$, $0 \leq c \leq \binom{r}{2}$, does there exist an $(r,c)$-constant graph?

The \textit{spectrum} of $r$, denoted by $\mathsf{spec}(r)$, is the set of all such integers $c$ such that an $(r,c)$-constant graph exists. The following theorem shows that $\mathsf{spec}(r)$ contains nearly all $c$ in the interval $\left[0, \binom{r}{2}\right]$. 

\begin{theorem}[Existence of $(r,c)$-constant graphs]\label{RC_Const_Exists}
	Let $r \in \mathbb{N}$. 
	\begin{enumerate}[i.]
		\item For every integer $c$ such that $0 \leq c \leq \dfrac{r^2}{2} - 5r^{\frac{3}{2}}$, $c \in \mathsf{spec}(r)$.
		\item Suppose $k \in \mathbb{N}$ and $r \geq 3k$. Then $\binom{r}{2} - k \notin \mathsf{spec}(r)$.
	\end{enumerate}
\end{theorem}
\begin{proof}
	First recall that for every $r,c\in \mathbb{N}$ such that $0 \le c \le r^2/2-5r^{3/2}$,
	there are positive integers $a_1,\ldots,a_s$ summing to $r$ such that $\sum_{j=1}^s \binom{a_j}{2}=c$;
	see \cite{reznick-1989} and \cite{CARO2023167}, Lemma 3.4.
	Consider the Cartesian product $G = \square_{j = 1}^s K_{a_j + 1}$. Then $G$ is $r$-regular with $r = \sum_{j=1}^s a_j$, and $e(v) = \sum_{j=1}^s \binom{a_j}{2} = c$, so (i) follows.
	
	We proceed to prove (ii). Suppose the contrary, and let $H$ be an $r$-regular graph where each vertex $v$ has $e(v) = \binom{r}{2} - k$. Then every vertex in $N(v)$ has at least $r - k$ neighbours in $N[v]$. 
	
	Let $x$ be a vertex in $N(v)$ with at most $r - 1$ neighbours in $N[v]$. Observe that the $k$ non-edges in the subgraph induced by $N[v]$ span at most $2k$ vertices. Hence there are $r+1-2k \geq k+1$ vertices in $N[v]$, including $v$, whose $r$ neighbours are all in $N[v]$ and therefore they are all neighbours of $x$. 
	
	Now consider $N[x]$, which contains some vertex $w$ not in $N[v]$, since $x$ has at most $r-1$ neighbours in $N[v]$. Then this vertex $w$ is not adjacent to at least $k+1$ vertices in $N[x] \cap N[v]$ (those having degree $r$ in $N[v]$). Therefore $w$ has at most $r - k - 1$ neighbours in $N[x]$. Consequently, $e(x) \leq \binom{r}{2} - k - 1$, a contradiction.
\end{proof}

\subsection{Flipping intervals and $k$-flip sequences}

The usefulness of $(r,c)$-constant graphs stems from the fact that together with CCP, they serve as the building blocks for long interval flips. In particular, we obtain a sufficient condition for $k$-flip  sequences.

\begin{theorem}\label{intervalFlippingThm}
Let $b \in \mathbb{N}$.
\begin{enumerate}[i.]
	\item If $b \geq 101$ then $\left[b, b + \left\lfloor \frac{1}{4} \big(b^2 - 10b^{\frac{3}{2}}\big)\right\rfloor\right]$ is a flipping interval.
	\item If $b \geq 3$ then $\left[b, 2b - 2\right]$ is a flipping interval.
\end{enumerate}	
\end{theorem}

\begin{proof}
	For (i), consider the interval $\left[b,b+k\right]$ where $k = \lfloor (b^2 - 10b^{\frac{3}{2}})/4\rfloor$. Since $b \geq 101$, we have $k \geq 12$.
	Set $M_1 = \lfloor b^2/2-5b^{3/2} \rfloor$.  For  $1 \le j \le k$, set $H_j$ to be a
$(b+j-1 , M_1 - 2(j-1) )$-constant graph which exists by Theorem \ref{RC_Const_Exists} and observe
$M_1 \ge 2k \ge 2(j-1)$  for $1 \le j \le k$. Consider the CCP graph $G = \square_{j=1}^k H_j$ where $H_j$ is  coloured using colour $j$. Then by Lemma \ref{CCPL1}, $G$ is a $(b, b+1, \dots, b+k)$-flip graph and therefore (i) follows.

We now prove (ii). Consider the triple $(b+j, b+1-j, 2j)$ where $0 \leq j \leq b-2$. Consider a regular bipartite graph $H$ of degree $\sum_{j=1}^{b-2} 2j = 
(b-1)(b-2)$, where a $2j$-factor of $H$ is coloured $j$ for $1 \leq j \leq b-2$.

We may take $H$ to be $K_{n,n}$ where $n = (b-1)(b-2)$. Now consider the CCP graph $G = H \ \square \ \left(
\square_{j=1}^{b-2} K_{b+1-j}\right)$ where $K_{b+1-j}$ is coloured $j$ for $1 \leq j \leq b-2$.

In every vertex $v$ of $G$, for $1 \leq j \leq b-2$, $\deg_j (v) = b + j$ which is increasing and $e_j [v] = \binom{b+1-j}{2} + 2j$, which is decreasing. Hence $G$ is a $(b, b+1, \dots, 2b - 2)$-flip graph.
\end{proof}

We note that for $b \leq 107$ the interval $[b, 2b-2]$ contains 
$[b, b + \lfloor (b^2 - 10b^{\frac{3}{2}})/4\rfloor]$.

\begin{corollary}[Sufficient condition for $k$-flip sequences]\label{sufficientKFlip}
	Suppose that $k \geq 2$. Let $3 
	\leq a_1 < a_2 < \dots < a_k$ be a sequence of integers such that either $a_k \leq 2a_1 - 2$ or $a_k \leq a_1 + \left\lfloor \frac{1}{4} \Big((a_1)^2 - 10(a_1)^{\frac{3}{2}}\Big)\right\rfloor$, then $(a_1, \dots, a_k)$ is a $k$-flip sequence.
\end{corollary}

\begin{proof}
	In both cases, the sequence $(a_1, \dots, a_k)$ is a subsequence of a flipping interval in Theorem \ref{intervalFlippingThm} and hence a $k$-flip sequence, as we can consider from the construction of (i) or (ii) in Theorem \ref{intervalFlippingThm} the edge-induced subgraph by the colours corresponding to the subsequence.
\end{proof}

Observe that the existence of $(r, c)$-constant graphs can be used to construct a wider class of graphs than just flip graphs. Namely $k$-edge coloured graphs, with prescribed colour degrees and edge-coloured neighbourhood sizes, in which for all vertices $v$, $\deg_j (v) = r_j$ and $e_j [v] = m_j$ for some positive integers $r_j, m_j$ satisfying $m_j \geq r_j$, $1 \leq j \leq k$.
To achieve this, we only need to assure the existence of $(r_j, c_j)$-constant graphs $H_j$ such that $m_j = r_j + c_j$ for $1 \leq j \leq k$, and then consider the CCP graph $\square_{j=1}^k H_j$.
\section{Existence of $t$-neighbourhood flip graphs}

We have so far considered flip colourings with regards to the immediate neighbourhood of a vertex. A natural extension of this problem, as outlined in Problem 5 in the introduction, is to consider the neighborhood consisting of all vertices at a distance (at most) $t$ from $v$. 

Before proceeding further, we require an adaptation of our existing notation. Let $k \in \mathbb{N}$ and let $G$ be a graph with an edge-coloring from $\{1, \dots, k\}$. As before, $E(j)$ is the set of edges coloured $j \in \{1, \dots, k\}$ in $G$. 
\begin{enumerate}
	\item For two vertices $u$ and $v$, $d(u,v)$ is the \textit{distance} between $u$ and $v$, i.e., the length of a shortest path between the two vertices.
	\item For a vertex $v$ and $t \in \mathbb{N}$, $N_t(v) = \{u \in V(G) \colon 1 \leq d(u,v) \leq t\}$ is its \textit{open} $t$-neighbourhood.
	\item For a vertex $v$ and $t \in \mathbb{N}$, $N_t[v] = N_t(v) \cup \{v\}$ is its \textit{closed} $t$-neighbourhood.
	\item For a vertex $v$ and $t \in \mathbb{N}$, $e_{j, t} (v) = \left|E(j) \cap E\left(N_t (v)\right)\right|$.
	\item For a vertex $v$ and $t \in \mathbb{N}$, $e_{j, t} [v] = \left|E(j) \cap E\left(N_t [v]\right)\right|$.
\end{enumerate} 

By $\Delta(G)$ and $g(H)$ we shall denote, as usual, the maximum degree of $G$ and the girth (shortest cycle length) of $G$, respectively. We now extend our general flip colouring problem to $t$-neighbourhoods: Given a graph $G = (V, E)$, and $k, t \in \mathbb{N}$ such that $k \geq 2$, does there exist an edge-colouring $f \colon E(G) \to \{1, \dots, k\}$ such that for every vertex $v$
\begin{itemize}
	\item $\deg_j (v) > \deg_i (v)$ for $1 \leq i < j \leq k$ (forcing global majority $e_j (G) > e_i (G)$),
	\item for $1 \leq s \leq t$, $e_{j,s} [v] < e_{i,s} [v]$ for $1 \leq i < j \leq k$ (forcing opposite majority order up to a distance $t$ from $v$, with respect to the global $e_j (G)$ and the local $\deg_j (v)$).
\end{itemize}

If such a colouring exists, then $G$ is said to be a $([t], k)$-flip graph (with respect to $f$). When we do not concern ourselves with the number of colours used, we simply say that $G$ is a $[t]$-flip graph.

As before, we shall mostly consider a more restricted version of this problem, where for every $j \in \{1, \dots, k\}$, the edge set $E(j)$ spans a regular subgraph of degree $a_j$, where $a_1 < a_2 < \dots < a_k$ and resulting in a $([t], (a_1, \dots, a_k))$-flip graph. 

We shall restrict to the case of two colours, demonstrating the existence of $([t], 2)$-flip graphs from two different perspectives, namely through Cayley graphs and packing arguments. 

\subsection{Constructing $\left([t], 2\right)$-flip graphs through Cayley graphs}

This subsection is devoted to constructing of $([t], 2)$-flip graphs through Cayley graphs.

Let $\Gamma$ be an abelian group and let $S$ be a subset of $\Gamma$ such that $S$ is inverse-closed and does not contain the identity element. Recall that the \textit{Cayley graph} $\cay{\Gamma}{S}$ is a graph with vertex set $\Gamma$ and edge set $\{\{g,gs\} \colon g \in \Gamma,~ s \in S\}$. The set $S$ is often termed as the \textit{connecting set} of the Cayley graph. 

By $\mathbf{0}$ we shall denote the all-zeros vector in $\mathbb{Z}_2^n$. Given any $i \in \{1, \dots, n\}$, let $\mathbf{e}_i$ be the vector in $\mathbb{Z}_2^n$ which is $1$ at the $i^\textrm{th}$ position, and $0$ everywhere else. Given any $1 \leq s \leq n$ and $0 \leq j \leq n - s$, we denote by $W_{s, j}$ the set of all binary vectors with the first $s$ entries all zero and exactly $j$ nonzero subsequent entries, namely
$$
W_{s, j} = \left\{\mathbf{w} \in \mathbb{Z}_2^n \colon \mathbf{w} \cdot \left(\sum_{i=1}^s \mathbf{e}_i\right) = 0 \ \wedge \ \mathbf{w} \cdot \left(\sum_{i=s+1}^n \mathbf{e}_i\right) = j
\right\}\,.
$$

Having established our working notation, we proceed to prove our main result for this subsection.

\begin{theorem}\label{CayleyTFlip}
	Let $s, t \in \mathbb{N}$, $t < s$. There exists a $\left([t], (2^s - 1, 2^s)\right)$-flip graph.	
\end{theorem}

\begin{proof}
	Let $n = 2^s + s$. Consider $B = \textrm{span}\{\mathbf{e}_1, \dots, \mathbf{e}_s\} \backslash \{\mathbf{0}\}$ and notice that $|B| = 2^s - 1$. Let $R = \{\mathbf{e}_{s+1}, \dots, \mathbf{e}_n\}$; by our choice of $n$, we have that $|R| = 2^s$. Consider $\cay{\mathbb{Z}_2^n}{R \cup B}$ with the edge-colouring $f \colon E \rightarrow \{1, 2\}$ such that given $v \in \mathbb{Z}_2^n$ and $\alpha \in R \cup B$: $$f\left(\{v, \alpha v\}\right) = 
	\begin{cases}
		1, & \alpha \in B \\
		2, & \alpha \in R
	\end{cases}$$
	
	Note that $R \cup B$ spans $\mathbb{Z}_2^n$,
	$R \cup B$ is inverse-closed, and $R \cap B = \emptyset$. Hence, $f$ is well-defined.
	
	Since a Cayley graph is vertex-transitive, it suffices to consider a single vertex. Consider $\mathbf{0} \in \mathbb{Z}_2^n$; by the edge-colouring $f$, we have that $$\deg_1 (\mathbf{0}) = 2^s - 1 < 2^s = \deg_2 (\mathbf{0})\,.$$
	
	Now, $N_1 (\mathbf{0}) = R \cup B$; we will add each of $R$ and $B$ to $N_1 (\mathbf{0})$ so that we find the vertices in $N_2 (\mathbf{0})$. Since $B \cup \{\mathbf{0}\}$ is a vector space, then in particular $B + B = B$. On the other hand, adding $B$ to $R$ results in the set $B + W_{s, 1}$. Hence, $(R \cup B) + B = B + W_{s, 1}$. Since neither $R$ nor $B$ includes $\mathbf{0}$, it follows that $(R \cup B) \cap (B + W_{s, 1}) = \emptyset$ and therefore $B + W_{s, 1} \subseteq N_2 (\mathbf{0})$. Likewise, adding $R$ to $B$ gives $B + W_{s, 1}$ once more and $R + R = W_{s,2}$. Therefore, $$N_2 (\mathbf{0}) = \left(B + W_{s, 1}\right) \ \dot\cup \ W_{s, 2}$$ and repeating the above argument for $1 \leq j < t$, we get that: $$N_{j+1} (\mathbf{0}) = \left(B + W_{s, j}\right) \ \dot\cup \ W_{s, j + 1}.$$
	
	Note that for $\mathbf{w} \in W_{s, j}$, $B + \mathbf{w}$ is a clique isomorphic to $K_{2^s - 1}$ since $\mathbf{w}$ is not in the span of $B$. More so, for $\mathbf{w}_1, \mathbf{w}_2 \in W_{s, j}$ such that $\mathbf{w}_1 \neq \mathbf{w}_2$, we have that $(B + \mathbf{w}_1) \cap (B + \mathbf{w}_2) = \emptyset$. Therefore, $(2^s - 1)\binom{n - s}{j} = (2^s - 1)\binom{2^s}{j}$ edges coloured $1$ arise between $W_{s,j}$ and $B + W_{s,j}$. Observe also that $(R + N_j (\mathbf{0})) \cap N_j (\mathbf{0}) = \emptyset$, and therefore the subgraph induced by $N_j (\mathbf{0})$ contains no edges coloured $2$.
	
	The edges coloured $1$ between $N_j (\mathbf{0})$ and $N_{j+1} (\mathbf{0})$ arise by adding $B$ to $W_{s, j}$ and therefore by our previous remark there are $(2^s - 1)\binom{2^s}{j}$ such edges. Meanwhile, the edges coloured $2$ between $N_j (\mathbf{0})$ and $N_{j+1} (\mathbf{0})$ arise by adding $R$ to $N_j (\mathbf{0})$. Adding $R$ to $W_{s, j}$ results in $(2^s - j) |W_{s, j}| = (2^s - j)\binom{2^s}{j}$ edges coloured $2$ between $N_j (\mathbf{0})$ and $N_{j+1} (\mathbf{0})$. On the other hand, adding $R$ to $B + W_{s, j-1}$ maps each clique in $B + W_{s, j-1}$ to a total of $2^s - j + 1$ cliques in $B + W_{s, j}$, with a perfect matching between every such pair of cliques. Therefore, there are an additional $(2^s - 1)(2^s - j + 1)\binom{2^s}{j-1}$ edges coloured $2$ between $N_j (\mathbf{0})$ and $N_{j+1} (\mathbf{0})$.
	
	By our previous remark, the subgraph induced by $N_{j+1} (\mathbf{0})$ contains no edges coloured $2$ and hence it follows that: 
	\begin{equation} \label{CayleyEqn1}
		e_{2, j+1} [\mathbf{0}] = e_{2, j} [\mathbf{0}] + (2^s - 1)(2^s - j + 1)\binom{2^s}{j - 1} + (2^s - j)\binom{2^s}{j}
	\end{equation}
	while between the vertices in $N_{j+1} (\mathbf{0})$ there are $\binom{n - s}{j} = \binom{2^s}{j}$ cliques coloured $1$ which are isomorphic to $K_{2^s - 1}$ and therefore:
	\begin{equation} \label{CayleyEqn2}
		e_{1, j+1} [\mathbf{0}] = e_{1, j} [\mathbf{0}] + \binom{2^s - 1}{2}\binom{2^s}{j} + (2^s - 1)\binom{2^s}{j} = e_{1, j} [\mathbf{0}] + \binom{2^s}{2}\binom{2^s}{j}\,.
	\end{equation}
	Now,
	\begin{align*}
		(2^s - 1)(2^s - j + 1)\binom{2^s}{j - 1} + (2^s - j)\binom{2^s}{j} & = (2^s - 1)j \binom{2^s}{j} + (2^s - j)\binom{2^s}{j} \\
		&\leq (2^s - 1)(j+1)\binom{2^s}{j}\\
		&< \binom{2^s}{2}\binom{2^s}{j}
	\end{align*}
where the last inequality follows from $j + 1 \leq t \leq s - 1 < \frac{2^s}{2}$.
	Consequently, from (\ref{CayleyEqn1}) and (\ref{CayleyEqn2}), we have for $1 \leq j < t$ that if $e_{2, j} [\mathbf{0}] < e_{1, j} [\mathbf{0}]$, then
	\begin{equation} \label{CayleyEqn3}
		e_{2, j+1} [\mathbf{0}] < e_{1, j+1} [\mathbf{0}]\;.
	\end{equation}
	Hence, it only remains to show that $e_{1,1} [\mathbf{0}] > e_{2,1} [\mathbf{0}]$.
	By choice of $R$ and $B$, we have that the vertices of $B$ in $\cay{\mathbb{Z}_2^n}{R \cup B}$ induce the complete graph $K_{2^s - 1}$ and therefore $$e_{1,1} [\mathbf{0}] = (2^s - 1) + \binom{2^s - 1}{2} = \binom{2^s}{2}$$ while the the vertices in $R$ are all linearly independent and hence $e_{2,1} [\mathbf{0}] = 2^s$. Hence, indeed, $e_{1,1} [\mathbf{0}] > e_{2,1} [\mathbf{0}]$.
\end{proof} 

\subsection{Constructing $([t], 2)$-flip graphs through packings}

In this subsection we construct $([t], 2)$-flip graphs using two classical graph theoretic results, concerned with the existence of $r$-regular graphs with large girth, and with graph packings.

\begin{theorem}[Erd\H{o}s-Sachs \cite{10.1007/s00493-014-2897-6, Exoo2008, LAZEBNIK1995275}]\label{ErdosSachs}
	Given integers $r \ge 2$ and $k \ge 3$, there are infinitely many connected $r$-regular graphs with girth at least $k$. 
\end{theorem}

\begin{theorem}[Sauer-Spencer-Catlin \cite{Catlin1, SAUER1978295}]\label{SauerSpencerCatlin}
	Let $G$ and $H$ be two graphs on $n$ vertices, such that $2 \Delta (G) \Delta (H) < n$. Then there exists a packing of $G$ and $H$ into an $n$ vertex set, with no overlapping edges.
\end{theorem}

A rooted tree of is $(j,b)$-\textit{perfect} if every internal vertex has $b$ children and all leaves are at distance $j$ from the root.

\begin{theorem} \label{PackingTFlip}
	Let $t \in \mathbb{N}$. There exist $([t], 2)$-flip graphs.
\end{theorem}

\begin{proof}
Let $b, r \in \mathbb{N}$ such that for some $q \in \mathbb{N}$, $q \geq 2$, we have that $(q+1)b \geq r \geq 2b + 1$ and $b \geq 2(q+3)^t$. We will construct a $\big([t], (2b, r)\big)$-flip graph. 

Suppose that $G^\ast$ and $H^\ast$ are connected graphs such that $G^\ast$ is $r$-regular with girth $g(G^\ast) > 2((q+3)b)^t$, and $H^\ast$ is $b+1$ regular with sufficiently large girth. Note that $L(H^\ast)$, the line graph of $H^*$, is $2b$-regular. We shall assume subsequently that $G^\ast$ and $H^\ast$ are as large as necessary. The existence of such graphs $G^\ast$ and $H^\ast$ is guaranteed by Theorem \ref{ErdosSachs}.

Let $p, p^\prime \in \mathbb{N}$ such that $2\Delta(G^\ast)\Delta(L(H^\ast)) < p \left|V(G^\ast)\right| = p^\prime \left|V\left(L(H^\ast)\right)\right|$. Then let $G = pG^\ast$ and $H = p^\prime L(H^\ast)$.
Since $G$ and $H$ are the union of disjoint copies of $G^\ast$ and $L(H^\ast)$ respectively, we have that $\Delta(G) = \Delta(G^*)$ and $\Delta(H) = \Delta(L(H^\ast))$. Furthermore, $g(G) = g(G^\ast)$. We will colour the edges of $H$ using $1$ and the edges of $G$ using $2$. 

Consider the vertex $v=\{x, y\}$ in $H$ (so $xy$ is an edge of some copy of $H^*$). Since the girth of $H^\ast$ is sufficiently large, and in particular much larger than $t$, it follows that in $H^\ast$, $x$ and $y$ are roots of two disjoint copies of a $(j+1,b)$-perfect tree $T$, for $1 \leq j \leq t$. Joining these two trees by the edge $\{x, y\}$, the line graph of the resulting graph is two copies of some block graph with $(b+1)$-cliques, coalesced at the vertex $\{x, y\}$. The number of $(b+1)$-cliques is, by virtue of $T$ being a $(j+1,b)$-perfect tree, $$2 \left(\sum_{i=0}^{j-1} b^i\right) = 2\left(\dfrac{b^j - 1}{b-1}\right)$$ and consequently, noting that all the edges in $H$ are coloured $1$, we have that 
\begin{equation}\label{e:ej1h}
e^{H}_{j, 1} [v] = 2\left(\dfrac{b^j - 1}{b-1}\right)\binom{b+1}{2} = \dfrac{(b^{j+1} - b)(b+1)}{b-1} > b^{j+1}
\end{equation} for $1 \leq j \leq t$.

Now, these two graphs $G$ and $H$ can be packed by Theorem \ref{SauerSpencerCatlin} into a graph $Q$ with no overlapping edges, while preserving their edge colourings. By this packing, $Q$ is $r + 2b$ regular, where every vertex has $2b$ incident edges coloured $1$ and $r$ incident edges coloured $2$. We will show that for any vertex $v$, $e_{j, 1} [v] > e_{j, 2} [v]$ in $Q$ for $1 \leq j \leq t$, and hence $Q$ is a $([t], (2b, r))$-flip graph.

We first compute an upper bound for $|N_j [v]|$ for any $v \in V(Q)$, observing that as $Q$ is $r+2b$ regular, and $G^\ast$ and $H^\ast$ are connected and can be arbitrarily large, it follows that $N_{j-1}[v]$ is strictly contained in $N_j [v]$ for $2 \leq j \leq t$. Firstly observe that for an $s$-regular graph we have that 
$$
	|N_j [v]| \leq 1 + s\sum_{i=0}^{j-1}(s-1)^i < 2s^j\,.
$$
Since $Q$ is $r + 2b$ regular and $r \leq (q+1)b$, by the inequality it holds for  $1 \leq j \leq t$ that
$$
	|N_j [v]| < 2(r+2b)^j \leq 2((q+3)b)^j\,.
$$
Due to the girth condition on $G$, we have that for $1 \leq j \leq t$, $g(G) > |N_j [v]|$. Hence, the subgraph in $Q$ induced by the edges coloured $2$ in $N_j [v]$ is acyclic, and therefore $$e^Q_{j, 2} [v] < |N_j [v]| < 2((q+3)b)^j\,,$$ while by \eqref{e:ej1h} $$e^Q_{j,1} [v] \geq e^H_{j, 1} [v] \geq b^{j+1}\,.$$

Therefore, for $1 \leq j \leq t$, to flip the majority in the $j^{\textrm{th}}$ neighbourhood of $v$ we require that $2((q+3)b)^j \leq b^{j+1}$, which simplifies to $2(q+3)^j \leq b$, which is the case since $2(q+3)^t \leq b$ and $j \leq t$. Hence $Q$ is a $([t], (2b, r))$-flip graph. 
\end{proof}
\section{Concluding remarks and open problems}

We have provided an in-depth treatment of most of the problems mentioned in the introduction, yet, nonetheless, several open problems remain. In particular, the complexity aspect of flip sequences and flip graphs (Problems 6). We next summarise a few additional open problems.

We have seen a comprehensive treatment of the two-colour case in Section 3.
However, with it still remains open the determination of $h(b, r)$, the smallest order of an $(b,r)$-flip graph.
\begin{problem*}
Determine $h(b,r)$ or at least obtain a nontrivial lower bound.
\end{problem*}
For three or more colours, Problem 2 remains entirely open.

Regarding three colors, in Theorem \ref{3colour_bound} we have seen a necessary condition for a sequence $(a_1, a_2, a_3)$ to be a $3$-flip sequence. In view of this theorem, it is of interest to find constructions of $3$-flip sequences $(a_1, a_2, a_3)$ with as large as possible a constant $c$ such that $a_3 = c (a_1)^2$, where we have seen that $\frac{1}{2} \leq c \leq 2$. 

\begin{problem*}
Determine the supremum over all constants $c$ such that there exist infinitely many $3$-flip sequences $(a_1, a_2, a_3)$ satisfying $a_3 \ge c(a_1)^2$.
\end{problem*}

Whilst for two and three colours we have a necessary condition for flip sequences, namely that the largest colour-degree is at most quadratic in the smallest colour-degree, we have proved in Theorem \ref{arbitraryGapsThm} that there is no such condition for $k \geq 4$ colours.
In fact, we have shown that there exists some $m(k) \in \mathbb{N}$ such that given any $N \in \mathbb{N}$, there is a $k$-flip sequence $(a_1, \dots, a_k)$ where $a_1 = m(k)$ and $a_k > N$. Explicitly from our construction we establish a linear upper bound on $m(k)$.

\begin{problem*}
Let $k \in \mathbb{N}, k \geq 4$. What is the minimum value of $m(k) \in \mathbb{N}$ such that for all $N \in \mathbb{N}$, there is a $k$-flip sequence $(m(k), a_2, \dots, a_k)$ where $a_k > N$?
\end{problem*}

We have explored the relationship between the smallest and largest colour-degrees, however for four colours and higher, in light of Theorem \ref{arbitraryGapsThm}, it is worth exploring the relationship between the other intermediate colour-degrees and the largest colour-degree. We pose the following problem. 

\begin{problem*}
For $k \ge 4$, what is the largest integer $q(k)$, $q(k) < k$, such that there exists some $h(k) \in \mathbb{N}$ and for all $N \in \mathbb{N}$, there is a $k$-flip sequence $(a_1, \dots,  a_k)$ where $a_{q(k)} = h(k)$ and $a_k > N$?
\end{problem*}
\noindent Our work in this paper establishes that $q(k) \ge 1$ for all $k \geq 4$.

Given the demonstrated usefulness of $(r, c)$-constant graphs, it is of interest to advance our knowledge concerning $\mathsf{spec}(r)$.

\begin{problem*}
	Determine $\mathsf{spec}(r)$, or at least improve upon the bounds given in Theorem \ref{RC_Const_Exists} for membership and non-membership in $\mathsf{spec}(r)$.
\end{problem*}  

We note that the graphs constructed in the proof of Theorem \ref{intervalFlippingThm} are substantially large, as are the graphs constructed to demonstrate the existence of $(r, c)$-constant graphs. It is therefore of interest to find smaller $(r,c)$-constant graphs, especially in light of Problem 2. 

\begin{problem*}
	Find lower and upper bounds to $$g(r, c) = \min\left\{|V(G)| \colon G \ \mathrm{is} \ \mathrm{an} \ (r,c)\mathrm{-constant} \ \mathrm{graph}\right\}\,.$$
\end{problem*}

Lastly, Section 6 dealt with the extension of the flip problem to the $t^\mathrm{th}$ neighbourhood. In particular we illustrated two distinct constructions for the case of two colours. Problems 1-5 naturally extend to this generalisation and remain to be studied. In particular, define $b(t) = \min\left\{ b \colon ([t], (b, r))\mathrm{-flip} \ \mathrm{graph} \ \mathrm{exists}\right\}$. Observe that for $t \geq 2$,
the proof of Theorem \ref{CayleyTFlip} gives $b(t) \leq 2^{t+1} - 1$.

\begin{problem*}
For $t \in \mathbb{N}, t \geq 2$ determine $b(t)$.
\end{problem*}

\bibliographystyle{plain}
\bibliography{flip_graphs_bibliography}

\begin{thebibliography}{10}

\bibitem{ABDULLAH20151}
M.A. Abdullah and M.~Draief.
\newblock Global majority consensus by local majority polling on graphs of a
  given degree sequence.
\newblock {\em Discrete Applied Mathematics}, 180:1--10, 2015.

\bibitem{blass1980trees}
A.~Blass, F.~Harary, and Z.~Miller.
\newblock Which trees are link graphs?
\newblock {\em Journal of Combinatorial Theory, Series B}, 29(3):277--292,
  1980.

\bibitem{CARO2023167}
Y.~Caro, J.~Lauri, and C.~Zarb.
\newblock The feasibility problem for line graphs.
\newblock {\em Discrete Applied Mathematics}, 324:167--180, 2023.

\bibitem{caro2018effect}
Y.~Caro and R.~Yuster.
\newblock The effect of local majority on global majority in connected graphs.
\newblock {\em Graphs and Combinatorics}, 34(6):1469--1487, 2018.

\bibitem{Catlin1}
P~Catlin.
\newblock Embedding subgraphs under extremal degree conditions.
\newblock {\em PhD thesis}, 1976.

\bibitem{chebotarev2023power}
P.l Chebotarev and D.~Peleg.
\newblock The power of small coalitions under two-tier majority on regular
  graphs.
\newblock {\em Discrete Applied Mathematics}, 340:239--258, 2023.

\bibitem{conder2021parameters}
M.~Conder, J.~Schillewaert, and G.~Verret.
\newblock Parameters for certain locally-regular graphs.
\newblock {\em arXiv preprint arXiv:2112.00276}, 2021.

\bibitem{10.1007/s00493-014-2897-6}
X.~Dahan.
\newblock Regular graphs of large girth and arbitrary degree.
\newblock {\em Combinatorica}, 34(4):407–426, aug 2014.

\bibitem{erdos1959graph}
P.~Erd\H{o}s.
\newblock Graph theory and probability.
\newblock {\em Canadian Journal of Mathematics}, 11:34--38, 1959.

\bibitem{Exoo2008}
G.~Exoo and R.~Jajcay.
\newblock Dynamic cage survey.
\newblock {\em The Electronic Journal of Combinatorics}, DS16:48 p., 2008.

\bibitem{FISHBURN1986165}
P.C. Fishburn, F.K. Hwang, and Hikyu Lee.
\newblock Do local majorities force a global majority?
\newblock {\em Discrete Mathematics}, 61(2):165--179, 1986.

\bibitem{furedi2013history}
Z.~F{\"u}redi and M.~Simonovits.
\newblock The history of degenerate (bipartite) extremal graph problems.
\newblock In {\em Erd{\H{o}}s centennial}, pages 169--264. Springer, 2013.

\bibitem{hajnal1989ramsey}
A.~Hajnal and P.~Erd\H{o}s.
\newblock Ramsey type theorems.
\newblock {\em Discrete Applied Mathematics}, 25:37--52, 1989.

\bibitem{https://doi.org/10.1002/jgt.3190090313}
J.~I. Hall.
\newblock Graphs with constant link and small degree or order.
\newblock {\em Journal of Graph Theory}, 9(3):419--444, 1985.

\bibitem{larrion2011small}
F.~Larri{\'o}n, M.A. Piza{\~n}a, and R.~Villarroel-Flores.
\newblock Small locally $nk_2$ graphs.
\newblock {\em Ars Comb.}, 102:385--391, 2011.

\bibitem{LAZEBNIK1995275}
F.~Lazebnik and V.~A. Ustimenko.
\newblock Explicit construction of graphs with an arbitrary large girth and of
  large size.
\newblock {\em Discrete Applied Mathematics}, 60(1):275--284, 1995.

\bibitem{lerman2016majority}
K.~Lerman, X.~Yan, and X.~Wu.
\newblock The ``majority illusion'' in social networks.
\newblock {\em PloS one}, 11(2):e0147617, 2016.

\bibitem{LISONEK1995153}
P.~Lisonêk.
\newblock Local and global majorities revisited.
\newblock {\em Discrete Mathematics}, 146(1):153--158, 1995.

\bibitem{reznick-1989}
Bruce Reznick.
\newblock The sum of the squares of the parts of a partition, and some related
  questions.
\newblock {\em Journal of Number Theory}, 33(2):199--208, 1989.

\bibitem{SAUER1978295}
N.~Sauer and J.~Spencer.
\newblock Edge disjoint placement of graphs.
\newblock {\em Journal of Combinatorial Theory, Series B}, 25(3):295--302,
  1978.

\bibitem{turan1941external}
P.~Tur{\'a}n.
\newblock On an external problem in graph theory.
\newblock {\em Mat. Fiz. Lapok}, 48:436--452, 1941.

\bibitem{turan1954theory}
P.~Tur{\'a}n.
\newblock On the theory of graphs.
\newblock In {\em Colloquium Mathematicum}, volume 1(3), pages 19--30, 1954.

\end{thebibliography}

\end{document}